\documentclass[12pt,a4paper]{article}
\usepackage{amsfonts,amsmath,amssymb}
\usepackage{amsthm}
\usepackage[cp866]{inputenc}
\usepackage[english]{babel}
\allowdisplaybreaks \theoremstyle{definition}
\newtheorem{Definition}{Definition}
\newtheorem{Remark}{Remark}
\newtheorem{Example}{Example}

\theoremstyle{plain}
\newtheorem{Theorem}{Theorem}

\newtheorem{Proposition}{Proposition}

\numberwithin{equation}{section} \textwidth 160mm \textheight 240mm
\title{Algebraic and analytic properties of quasimetric spaces with dilations\thanks{This
research was partially supported by Federal Target Grant
"Scientific and educational personnel of innovation Russia" for
2009-2013 (government contract No. P2224) and by the program
``Leading Scientific Schools'' (project N. NSh-5682.2008.1).}}
\author{Svetlana Selivanova, Sergei Vodopyanov}

\begin{document}

\date{}
\maketitle

\begin{abstract} We provide an axiomatic approach to the theory of local tangent
cones of regular sub-Riemannian manifolds and the
differentiability of mappings between such spaces. This axiomatic
approach relies on a notion of a dilation structure which is
introduced in the general framework of quasimetric spaces.
Considering quasimetrics allows us to cover a general case
including, in particular, minimal smoothness assumptions on the
vector fields defining the sub-Riemannian structure. It is
important to note that the theory existing for metric spaces can
not be directly extended to quasimetric spaces.


\end{abstract}

\section{Introduction}\label{intro}
We study algebraic and analytic properties of quasimetric spaces
endowed with dilations (roughly speaking, dilations are continuous
one-parameter families of contractive homeomorphisms given in a
neighborhood of each point).

Our work is motivated by investigation of metric properties of
Carnot-Carath\'{e}o\-dory spaces, also referred to as
sub-Riemannian manifolds which model nonholonomic processes and
naturally arise in many applications (see e. g.
\cite{as,bel,lan,bur,cdpt,gro,hor,fs,vk,mm,M,nsw,rs,Stein} and
references therein).

Let us first recall the ``classical'' definition of a
sub-Riemannian manifold. Given a smooth connected manifold
$\mathbb M$ of dimension $N$ and smooth ``horizontal'' vector
fields $X_1,\ldots, X_m\in C^{\infty}$ on $\mathbb M$ (where $m\leq
N$), it is assumed that these vector fields span, together with
their commutators, the tangent space to $\mathbb M$ at each point
(H\"{o}rmander's condition \cite{hor}). By Rashevski\v{\i}-Chow's
Theorem, any two points of $\mathbb M$ can be connected by a
horizontal curve and, therefore, there exists an intrinsic
sub-Riemannian metric $d_c$ on $\mathbb M$ defined as the infimum
over lengths of all horizontal curves.

 Recently discovered applications have lead to considering a
more general situation \cite{karm,vk,dan,vk3,vk4,vk5} when

1) a maximal possible reduction of smoothness of the vector fields
is made (see also \cite{bram,gre1,morbid});

2) instead the H\"{o}rmander's condition, a weaker one of a
``weighted'' filtration of $T\mathbb M$ (see Definition \ref{cc})
is assumed (see also \cite{fol,fs,gre1,nsw,Stein}).

Under these general assumptions, the intrinsic metric $d_c$ might
not exist, but a certain quasimetric (a distance function meeting
a generalized triangle inequality, see Definition \ref{km}) can be
introduced (see \cite{nsw} where various quasimetrics induced by
families of vector fields on $\mathbb R^N$ were studied).

 On the other hand, recent development of
analysis on general metric spaces has lead to the question of
describing the most general approach to the metric geometry of
sub-Riemannian manifolds. Among possible approaches is considering
metric spaces with dilations \cite{bel,bulf,bul_sr,fs}.

Motivated by these considerations, we extend the notion of a
dilation structure to quasimetric spaces and investigate local
properties of the obtained object.

In 1981 M. Gromov has defined  \cite{gro1,gro2} the tangent cone
to a metric space $({\mathbb X},d)$ at a point $x\in {\mathbb X}$
as the  limit of pointed scaled metric spaces $({\mathbb
X},x,\lambda\cdot d)$ (when $\lambda\to\infty$) w. r. t.
Gromov-Hausdorff distance. This notion generalizes the concept of
the tangent space to a manifold and is useful in the general
theory of metric spaces (see e. g. \cite{ber,bur,ch,ld,pet}), in
particular, Carnot-Carath\'{e}odory spaces \cite{mm,mit}.

A straightforward generalization of Gromov's theory would make no
sense for quasimetric spaces, see Remark \ref{cone_rem}. In
\cite{dan,smj} a convergence theory for quasimetric spaces with
the following properties was developed:

1) it includes the Gromov-Hausdorff convergence for metric spaces
as a particular case;

2) the limit is unique up to isometry for boundedly compact
quasimetric spaces;

3) it allows to introduce the notion of the tangent cone in the
same way as for metric spaces.

 In
\cite{smj}  the existence of the tangent cone (w. r. t. the
introduced convergence) to a quasimetric space with dilations is
proved (see Definition \ref{dil_struct}, Axioms (A0)~---(A3), and
Theorem \ref{exist_teor}). This statement contains as a particular
case a similar result by M. Buliga for metric spaces, see for
instance \cite{bulf}, where an axiomatic approach to metric spaces
with dilations is introduced. A similar approach was informally
sketched by A. Bellaiche \cite{bel}.

The main results of the present paper are Theorems \ref{dil_teor}
and \ref{cc_teor}. Theorem \ref{dil_teor} (cf. \cite{bul_contrac})
asserts that an additional axiom (A4) (saying that the limit of a
certain combination of dilations exists) allows to describe the
algebraic structure of the tangent cone: it is a simply connected
Lie group, the Lie algebra of which is graded and nilpotent.

In particular, this result allows to define the differential of a
mapping acting between two quasimetric spaces with dilations in
the same way as it is done in \cite{v} for Carnot-Carath\'{e}odory
spaces.  A brief comparison of this approach with
Margulis-Mostow's concept of differentiability \cite{mm} is given
below in Remark \ref{comp}.

Thus, Theorem \ref{dil_teor} allows to establish algebraic and
analytic properties of the considered space
 from
metric and topological assumptions only. In the present paper we
do not attempt to prove that axioms of a dilation structure
recover sub-Riemannian geometry when the underlying space is a
manifold (or which axioms should be added to prove this). But we
prove that

1) regular sub-Riemannian manifolds are examples of quasimetric
spaces with dilations (Theorem \ref{cc_teor});

2) the tangent cones to quasimetric spaces with dilations are the
same algebraic objects as for regular sub-Riemannian manifolds
(Theorem \ref{dil_teor}),

which can be viewed as a first step in this direction.

In our opinion, the proof of Theorem \ref{dil_teor} is interesting
in its own right. The main step is to apply a  theorem on local
and global topological groups due to A. I. Mal'tsev \cite{mal},
which helps to overcome difficulties concerned with investigation
of a local version of the Hilbert's Fifth Problem
\cite{hilb,gleas,mz}, see Remark \ref{bul_rem}. As an auxiliary
assertion we prove a generalized triangle inequality for local
groups endowed with (quasi)metrics and dilations (see Proposition
\ref{odnor}, Assertion 3)), which is of independent interest and
gives an alternative proof of a similar fact for  (global)
homogeneous groups ~\cite{fs}.

In Section \ref{cc_sec}, we describe regular
Carnot-Carath\'{e}odory spaces as the main example of quasimetric
spaces with dilations. In this case Axiom (A3) is just a local
approximation theorem, and (A4) is a consequence of estimates on
divergence of integral lines of the initial vector fields and the
nilpotentized ones.

In this paper we extend the approach to the subject given in our
short communication \cite{vs}.

We are grateful to Isaac Goldbring for a discussion on some
algebraic aspects of the subject under consideration
 (see Remark \ref{u_n})
and for the references \cite{olv,gol}. We  thank also the
anonymous referee for the careful reading of our paper,
interesting questions and references, as well as useful hints
concerning the presentation and exposition of our results.

 \section{Basic notions and preliminary results}

\begin{Definition}
\label{km} A {\it quasimetric space } $({\mathbb X},d_{\mathbb
X})$ is a topological space ${\mathbb X}$ with a quasimetric
$d_{\mathbb X}$. A {\it quasimetric } is a mapping $d_{\mathbb
X}:{\mathbb X}\times {\mathbb X}\to \mathbb R^+$ with the
following properties:

(1) $d_{\mathbb X}(u,v)\geq 0$; $d_{\mathbb X}(u,v)= 0$ if and
only if $u=v$ $($non-degeneracy$)$;

(2) $d_{\mathbb X}(u,v)\leq c_{\mathbb X}d_{\mathbb X}(v,u)$ where
$1\leq c_{\mathbb X}<\infty$ is a constant independent of $u,v\in
{\mathbb X}$ (generalized symmetry property);

(3) $d_{\mathbb X}(u,v)\leq Q_{\mathbb X}(d_{\mathbb
X}(u,w)+d_{\mathbb X}(w,v))$ where $1\leq Q_{\mathbb X}<\infty$ is
a constant independent of $u,v,w\in {\mathbb X}$ $($generalized
triangle inequality$)$;

(4) the function $d_{\mathbb X}(u,v)$ is upper semi-continuous on
the first argument.

\noindent If $c_{\mathbb X}=1$, $Q_{\mathbb X}=1$, then $({\mathbb
X},d_{\mathbb X})$ is a metric space.
\end{Definition}

\begin{Remark}\label{km_top}
Note that some authors introduce the notion of a quasimetric space
without assuming neither this space be topological nor the
quasimetric be continuous in any sense. Within such framework, the
quasimetric balls need not be open (see e. g. \cite{PS, christ,
hei}). However, due to a theorem by R. A. Mac\`{\i}as and C.
Segovia \cite{macseg}, any quasimetric $d$
 is equivalent to some other quasimetric $\tilde{d}$, the balls associated to
which are open (such a quasimetric looks like
$\rho(x,y)^{\frac{1}{\beta}}$, where $0<\beta\leq 1$ and
$\rho(x,y)$ is a metric) and, hence, define a topology.

In the present paper we study tangent cone questions. It is
important to note, that having the tangent cone to a (quasi)metric
space, one can say nothing about the existence of the tangent cone
to the space with an equivalent (quasi)metric, thus we would like
the balls defined by the initial quasimetric be open. For this
reason we add the upper-continuity condition (4) to the Definition
\ref{km} of a quasimetric space (as it is done e. g. in
\cite{Stein} for the case of $\mathbb R^n$). This condition
guarantees that the balls $B^{d_\mathbb X}(x,r)$ are open sets,
and that convergence w. r. t. the initial topology of $\mathbb X$
implies convergence w. r. t. the topology defined by $d_{\mathbb
X}$.

Actually, we can assume the initial topology on $\mathbb X$
coincide with the topology induced by the equivalent quasimetric
$\tilde{d}$. Then the topologies induced by $d$ and convergence w.
r. t. initial topology on $\mathbb X$ are  equivalent. Further we
always assume, w. l. o. g., this to hold.
\end{Remark}

We denote by $B^{d_\mathbb X}(x,r)=\{y\in \mathbb X\mid d_{\mathbb
X}(y,x)<r\}$ a ball centered at $x$ of radius $r$, w. r. t. the
(quasi)metric $d_{\mathbb X}$. The symbol $\bar{A}$ stands for the
closure of the set $A$. A (quasi)metric space ${\mathbb X}$ is
said to be {\it boundedly compact} if all closed bounded subsets
of ${\mathbb X}$ are compact.

\begin{Definition}
\label{dil_struct} Let $({\mathbb X},d)$ be a complete boundedly
compact quasimetric space and the quasimetric $d$ be continuous on
both arguments. The quasimetric space ${\mathbb X}$ is endowed
with a {\it dilation structure}, denoted as $({\mathbb
X},d,\delta)$, if the following axioms (A0)~--- (A3) hold.

(A0) For all $x\in {\mathbb X}$ and for $\varepsilon\in(0,1]$, in
some neighborhood $U(x)$ of $x$ there are homeomorphisms called
{\it dilations} $\delta_{\varepsilon}^x:U(x)\to
V_{\varepsilon}(x)$ and
$\delta_{\varepsilon^{-1}}^x:W_{\varepsilon^{-1}}(x)\to U(x)$,
where $V_{\varepsilon}(x)\subseteq
W_{\varepsilon^{-1}}(x)\subseteq U(x)$. The family
$\{\delta^x_{\varepsilon}\}_{\varepsilon\in(0,1]}$ is continuous
on $\varepsilon$ (w. r. t. the initial topology on $\mathbb X$,
see Remark \ref{km_top}, and the ordinary topology on $(0,1]$). It
is assumed that there exists an $R>0$ such that
$\bar{B}^d(x,R)\subseteq U(x)$ for all $x\in {\mathbb X}$, and
 for all $\varepsilon<1$ and $\tilde{r}>0$ with the property
 $\bar{B}^d(x,\tilde{r})\subseteq U(x)$ we have the inclusion
  $B^d(x,\tilde{r}\varepsilon)\subseteq
\delta_{\varepsilon}^xB^d(x,\tilde{r})\subset B^d(x,\tilde{r})$.

(A1) For all $ x\in {\mathbb X}$, $y\in U(x)$, we have
$\delta_{\varepsilon}^xx=x,\ \delta_1^x=\text{id},\
\lim\limits_{\varepsilon\to 0} \delta_{\varepsilon}^xy=x$.

(A2) For all $x\in {\mathbb X}$ and  $u\in U(x)$,
 we have
$\delta_{\varepsilon}^x \delta_{\mu}^x u
=\delta_{\varepsilon\mu}^xu$ provided that both parts of this
equality are defined.

(A3)  For any $x\in {\mathbb X}$, uniformly on $u,v\in
\bar{B}^{d}(x,R)$ there exists the limit
 \begin{equation}\label{a3}\lim\limits_{\varepsilon\to 0}
 \frac{1}{\varepsilon}d(\delta_{\varepsilon}^xu,\delta_{\varepsilon}^xv)=
d^x(u,v).\end{equation}

If the function $d^x:U(x)\times U(x) \to \mathbb R^+$ is such that
$d^x(u,v)=0$ implies $u=v$, then the dilation structure is called
{\it nondegenerate}.

If the convergence in $(A3)$ is uniform on $x$ in some compact
set, then the dilation structure is said to be  {\it uniform}.

If the following axiom (A4) holds, then we say that ${\mathbb X}$
is endowed with a {\it strong dilation structure}.

(A4) The limit of the value $\Lambda_{\varepsilon}^x(u,v)=
\delta_{\varepsilon^{-1}}^{\delta_{\varepsilon}^xu}\delta_{\varepsilon}^{x}v$
 exists:
\begin{equation} \label{lam}
\lim\limits_{\varepsilon\to 0}
\Lambda_{\varepsilon}^x(u,v)=\Lambda^x(u,v)\in B^d(x,R),
\end{equation}
This limit is uniform on  $x$ in some compact set and $u,v\in
B^d(x,r)$ for some $0<r\leq R$. See Proposition \ref{obl_opr}
regarding possible choices of $r$.
\end{Definition}

\begin{Remark} \label{bul_rem} These axioms of dilations are a slight
modification and simplification of those proposed in \cite{bulf}
for metric spaces. Essential for proving Theorem  \ref{dil_teor}
is that, in (A0), we require the continuity of dilations on the
parameter $\varepsilon$ which was missed in \cite{bulf}. Note also
that axioms (A1), (A2), (A4) do not depend on the quasimetric. The
condition $\lim\limits_{\varepsilon\to 0}
\delta_{\varepsilon}^xy=x$ informally states that the topological
space ${\mathbb X}$ is locally contractible.\end{Remark}

\begin{Example}In the case when ${\mathbb X}$ is a Riemannian manifold,
dilations can be introduced as homotheties induced by the
Euclidean ones. See \cite{bulf}--\cite{bul_sr_sym} for more
examples.\end{Example}

\begin{Remark}\label{balls} For a general (quasi)metric space
$(\mathbb X,d)$, the closure of a ball need not coincide with the
corresponding closed ball, only the inclusion
$\bar{B}^d(x,r)\subseteq \{y:d(y,x)\leq r\}$ holds. But, in the
case of a (quasi)metric space endowed with a dilation structure,
also the converse inclusion is true. Indeed, let $z\in
\{y:d(y,x)\leq r\}$ be such that $d(z,x)=r$; let
$z_n=\delta^x_{1-\varepsilon_n}z\in B^d(x,r)$, where
$\varepsilon_n\to 0$. Then $d(z_n,z)\to 0$, according to (A0),
(A1) and Remark \ref{km_top}, hence $z\in\bar{B}^d(x,r)$.
\end{Remark}

\begin{Remark}\label{top} By virtue of (A3) and continuity of $d(u,v)$,
the function $d^x(u,v)$ is continuous on both arguments. Further,
the functions $d^x$ and $d$ define the same topology on $U(x)$
(the equivalence of convergences induced by $d^x$ and $d$ can be
verified straightforwardly, using uniformity on $u,v$ in (A3))
and, hence, $(U(x),d^x)$ is boundedly compact.
\end{Remark}

Remark \ref{top} and Axiom (A3) imply

\begin{Proposition}\label{d_x} If $({\mathbb X},d,\delta)$ is a nondegenerate
dilation structure, then $d^x$ is a quasimetric on $B^d(x,R)$ with
the same constants $c_{\mathbb X}$, $Q_{\mathbb X}$ $($see $(2)$,
$(3)$ of Definition $\ref{km}$$)$ as for the initial quasimetric
$d$.
\end{Proposition}

In the same way as for  metric spaces \cite{bulf}, Axioms (A2),
(A3) imply

\begin{Proposition}\label{cone_prop} The function $d^x$ from Axiom $(A3)$
meets the cone property
    $$d^x(u,v)=\frac{1}{\mu}d^x(\delta_{\mu}^xu,\delta_{\mu}^xv)$$
for all $u,v\in B^d(x,R)$ and $\mu$ such that
$\delta_{\mu}^xu,\delta_{\mu}^xv\in B^d(x,R)$ $($in particular,
for all $\mu\leq 1$$)$.
\end{Proposition}

\begin{Proposition}\label{sig_inv}If $({\mathbb X},d,\delta)$ is a strong dilation structure
then the limits of the  expressions $\Sigma_{\varepsilon}^x(u,v)=
\delta_{\varepsilon^{-1}}^x\delta_{\varepsilon}^{\delta_{\varepsilon}^xu}v,
\
\operatorname{inv}^x_{\varepsilon}(u)=\delta_{\varepsilon^{-1}}^{\delta_{\varepsilon}^xu}x$
exist:
\begin{equation}\label{sig_inv_def}\lim\limits_{\varepsilon\to 0}
\Sigma_{\varepsilon}^x(u,v)=\Sigma^x(u,v)\in B^d(x,R),\
\lim\limits_{\varepsilon\to 0}
\operatorname{inv}^x_{\varepsilon}(u)=\operatorname{inv}^x(u)\in
B^d(x,R).\end{equation} These limits are uniform on $x$ in some
compact set and $u,v\in B^d(x,\hat{r})$.

Conversely, if the limits \ref{sig_inv_def} exist and are uniform,
then Axiom (A4) holds.
\end{Proposition}
\begin{proof}The assertion about the second limit follows from the fact that
$\operatorname{inv}^x_{\varepsilon}(u)=\Lambda_{\varepsilon}^x(u,x)$.
Easy calculations show that
$\Sigma^x_{\varepsilon}(u,v)=\Lambda_{\varepsilon}^{\delta_{\varepsilon}^xu}
(\text{inv}^x_{\varepsilon}u,v)$ from where, taking in account the
uniformity of convergence in (A4), the existence and uniformity of
the first limit follows.

Moreover, it is easy to see that
$\Sigma_{\varepsilon}^{\delta_{\varepsilon}^xu}
(\text{inv}^x_{\varepsilon}u,v)=\Lambda^x_{\varepsilon}(u,v),$
hence
\begin{equation}\label{del_sig}\Lambda^x(u,v)=\Sigma^x(\text{inv}^xu,v).\end{equation}
Therefore, from the existence and uniformity of the limits
\ref{sig_inv_def}, Axiom (A4) follows.
\end{proof}

Further we assume, w. l. o. g., that $\hat{r}=r$ (otherwise, take
the intersection of the corresponding balls), i. e. functions
$\Lambda^x$ and $\Sigma^x$ are defined on the same subset of
$U(x)\times U(x)$. The following proposition can be viewed as an
example of existence of one of the combinations from Proposition
\ref{sig_inv} (cf. the arguments of Bellaiche \cite{bel}, the last
section).

\begin{Proposition}\label{obl_opr} Let  $({\mathbb X},d,\delta)$ be
a uniform dilation structure. Then there are $r,\varepsilon_0>0$
such that for all $\varepsilon\in(0,\varepsilon_0]$, $u,v\in
B^d(x,r)$ the combination $\Sigma_{\varepsilon}^x(u,v)=
\delta_{\varepsilon^{-1}}^x\delta_{\varepsilon}^{\delta_{\varepsilon}^xu}v\in
U(x)$ from Proposition $\ref{sig_inv}$ is defined.
\end{Proposition}

\begin{proof}
Let $x^{\prime}=\delta_{\varepsilon}^xu$,
$x^{\prime\prime}=\delta_{\varepsilon}^{x^{\prime}}v$. To show the
existence of the combination $\Sigma_{\varepsilon}^x(u,v)\in U(x)$
it suffices to verify that $x^{\prime\prime}\in
W_{\varepsilon^{-1}}(x)$. Let us prove that, for suitable
$u,v,\varepsilon$, it is true that $x^{\prime\prime}\in
B^d(x,R\varepsilon)\subseteq W_{\varepsilon^{-1}}(x)$. It follows
from Proposition \ref{cone_prop} that
$d^x(x,x^{\prime})=d^x(x,\delta_{\varepsilon}^xu)=\varepsilon
d^x(x,u)$,
$d^{x^{\prime}}(x^{\prime},x^{\prime\prime})=\varepsilon
d^{x^{\prime}}(x^{\prime},v)$. Due to (A3), for any $\delta>0$
there is an $\varepsilon>0$ such that: if
$d^x(p,q)=O(\varepsilon)$, then $d^x(p,q)(1-\delta)\leq d(p,q)\leq
d^x(p,q)(1+\delta)$. Let $p=x$, $q=x^{\prime}$ and consider
arbitrary $r, R^x>0$ such that $B^d(x,r)\subseteq
B^{d^x}(x,R^x)\subseteq B^d(x,R)$ (such reals exist according to
Remark \ref{top}). For any $\delta>0$ there is an
$\varepsilon_0^{\prime}>0$ such that for $u\in B^d(x,r)$,
$\varepsilon\in(0,\varepsilon_0^{\prime}]$ we have
$d(x,x^{\prime})\leq\varepsilon R^x(1+\delta)$. Analogously, there
is an $\varepsilon_0^{\prime\prime}>0$ such that for $v\in
B^d(x,r)$, $\varepsilon\in(0,\varepsilon_0^{\prime\prime}]$  we
have $d(x^{\prime},x^{\prime\prime})\leq\varepsilon
R^{x^\prime}(1+\delta)$. Due to uniformity of the limit in (A3) we
can assume, w. l. o. g., that $R^x=R^{x^\prime}=\xi$. Let
$\varepsilon_0=\min\{\varepsilon_0^{\prime},\varepsilon_0^{\prime\prime}\}$.
The generalized triangle inequality implies
$d(x,x^{\prime\prime})\leq Q_{\mathbb
X}\left(d(x,x^{\prime})+d(x^{\prime},x^{\prime\prime})\right)\leq
2Q_{\mathbb X}\varepsilon\xi(1+\delta)$. To satisfy the desired
inequality $d(x,x^{\prime\prime})\leq R\varepsilon$ it suffices to
take an arbitrary $\xi< \frac{R}{2Q_{\mathbb X}}$ such that
$B^{d^x}(x,\xi)\subseteq B^d(x,R)$. Then an arbitrary number $r$
satisfying $B^d(x,r)\subseteq B^{d^x}(x,\xi)$ will be as desired.
\end{proof}

A {\it pointed $($quasi$)$metric space} is a pair $({\mathbb
X},p)$ consisting of a (quasi)metric space ${\mathbb X}$ and a
point $p\in {\mathbb X}$. Whenever we want to emphasize what kind
of (quasi)metric is on ${\mathbb X}$, we shall write the pointed
space as a triple $({\mathbb {\mathbb X}},p,d_{\mathbb X})$.

\begin{Definition}[\cite{dan, smj}]\label{km_conv_nc}
A sequence $({\mathbb X}_n,p_n,d_{{\mathbb X}_n})$ of pointed
quasimetric spaces {\it converges} to the pointed space $({\mathbb
X},p,d_{\mathbb X})$, if there exists a sequence of reals
$\delta_n\to 0$ such that for each $r>0$ there exist mappings
$f_{n,r}:B^{d_{{\mathbb X}_n}}(p_n,r+\delta_n)\to {\mathbb X},\
g_{n,r}:B^{d_{{\mathbb X}}}(p,r+2\delta_n)\to {\mathbb X}_n$ such
that

1) $f_{n,r}(p_n)=p, \ g_{n,r}(p)=p_n$;

2) $\operatorname{dis}(f_{n,r})<\delta_n,\
\operatorname{dis}(g_{n,r})<\delta_n;$

3) $\sup\limits_{x\in B^{d_{{\mathbb
X}_n}}(p_n,r+\delta_n)}d_{{\mathbb
X}_n}(x,g_{n,r}(f_{n,r}(x)))<\delta_n$.
\end{Definition}

Here $\operatorname{dis}(f)=\sup\limits_{u,v\in {\mathbb
X}}|d_{\mathbb Y}(f(u),f(v))-d_{\mathbb  X}(u,v)|$ is the {\it
distortion} of a mapping $f:({\mathbb X},d_{\mathbb X})\to
({\mathbb Y},d_{\mathbb Y})$ which characterizes the difference of
$f$ from an isometry.

\begin{Theorem}[\rm\cite{smj}]\label{km_teor}
1. Reduced to the case of metric spaces, the convergence of
Definition \ref{km_conv_nc} is equivalent to the Gromov-Hausdorff
one;

2) Let $(X,p),\ (Y,q)$ be two complete pointed quasimetric spaces,
each obtained as a limit  of the same sequence $(X_n,p_n)$ such
that the constants $\{Q_{X_n}\}$ are uniformly bounded:
$|Q_{X_n}|\leq C$ for all $n\in\mathbb N$. If $X$ is boundedly
compact, then $X$ and $Y$ are isometric.
\end{Theorem}

\begin{Remark}\label{conv_rem}
 Note that a
straightforward generalization of Gromov's theory to the case of
quasimetric spaces is, for various reasons, impossible. For
example, the Gromov-Hausdorff distance between two bounded
quasimetric spaces is equal to zero \cite{gre} and, thus, makes no
sense in this context. Besides that, in \cite{gro,bel} convergence
is first defined for compact spaces;  convergence of boundedly
compact spaces is defined as convergence of all (compact) balls.
For quasimetric spaces, this approach  would not yield uniqueness
of the limit up to isometry.
\end{Remark}

\begin{Definition}\label{cone_def} Let ${\mathbb {\mathbb X}}$ be a boundedly compact
(quasi)metric space, $p\in X$. If the limit of pointed spaces
$\lim\limits_{\lambda\to\infty}(\lambda {\mathbb
X},p)=(T_p{\mathbb X},e)$  exists (in the sense of Definition
\ref{km_conv_nc}), then
 $T_p {\mathbb X}$ is called the {\it tangent cone} to ${\mathbb X}$ at $p$.
Here $\lambda {\mathbb X}=({\mathbb X},\lambda\cdot d_{\mathbb
X})$; the symbol $\lim\limits_{\lambda\to\infty} (\lambda {\mathbb
X},p)$ means that, for any sequence $\lambda_n\to\infty$, there
exists $\lim\limits_{\lambda_n\to\infty}(\lambda_n {\mathbb X},p)$
which is
 independent of the choice of sequence $\lambda_n\to\infty$ as
 $n\to\infty$.

Any neighborhood $U(e)\subseteq T_p{\mathbb X}$ of the basepoint
element $e\in T_p{\mathbb X}$ is said to be a {\it local tangent
cone} to ${\mathbb X}$ at $p$.
\end{Definition}

\begin{Remark}\label{cone_rem}
Theorem \ref{km_teor} implies that, for complete boundedly compact
quasimetric spaces, the tangent cone is unique up to isometry,
i.~e. one should treat the tangent cone from Definition
\ref{cone_def} as a
 class of pointed quasimetric spaces isometric to each other. Note also that
  the tangent cone is completely defined by any (arbitrarily small) neighborhood
  of the point. More precisely, if $U$ is a neighborhood
of the point $p\in {\mathbb X}$ then the tangent cones of $U$ and
${\mathbb X}$ at $p$ are isometric. Moreover, the quasimetric
space $(T_p{\mathbb X},e)$ is a cone in the sense that it is
invariant under scalings, i. e. $(T_p{\mathbb X},e)$ is isometric
to $(\lambda T_p{\mathbb X},e)$ for all $\lambda>0$.
\end{Remark}

\begin{Theorem}[\rm\cite{smj}]\label{exist_teor} Let
$({\mathbb X},d,\delta)$ be a nondegenerate dilation structure.
Then $(U(x),x,d^x)$ is a local tangent cone to ${\mathbb X}$ at
$x$.
\end{Theorem}

Note that on the neighborhood $U(x)\subseteq \mathbb X$ we have
two (quasi)metric structures $d$ and $d^x$, thus it is natural to
denote the local tangent cone to $\mathbb X$ at $x$ as
$(U(x),d^x)$, not introducing any other underlying set for the
tangent cone.

 One of the main goals of
the present paper is to describe the algebraic properties of the
(local) tangent cone in the case when $({\mathbb X},d,\delta)$ is
a strong uniform nondegenerate dilation structure. Having only
axioms (A0)~--- (A3) we can say nothing substantial about this.

\section{Algebraic properties of the tangent cone}

\begin{Definition}[\rm\cite{pon,gol}]\label{loc_gr} A {\it local group}
is a tuple $({\mathcal G},e,i,p)$ where ${\mathcal G}$ is a
Hausdorff topological space with a fixed {\it identity element}
$e\in {\mathcal G}$ and continuous functions $i:\Upsilon\to
{\mathcal G}$ ({\it the inverse element function}), and
$p:\Omega\to {\mathcal G}$ ({\it the product function}) given on
some subsets $\Upsilon\subseteq {\mathcal G}$, $\Omega\subseteq
{\mathcal G}\times {\mathcal G}$ such that $e\in\Upsilon$,
$\{e\}\times {\mathcal G}\subseteq\Omega$, ${\mathcal
G}\times\{e\}\subseteq\Omega$, and for all $x,y,z\in {\mathcal G}$
the following properties hold:

1) $p(e,x)=p(x,e)=x$;

2) if $x\in\Upsilon$, then $(x,i(x))\in\Omega$,
$(i(x),x)\in\Omega$ and $p(x,i(x))=p(i(x),x)=e$;

3) if $(x,y),(y,z)\in\Omega$ and $(p(x,y),z),(x,p(y,z))\in\Omega$,
then $p(p(x,y),z)=p(x,p(y,z))$.

\end{Definition}

Assertions close to the next three  propositions can be found in
 \cite{bulf}, but in our consideration, some details are
 different. We include the proofs for the reader's convenience.

\begin{Proposition}\label{a4_loc_gr}
Let $({\mathbb X},d,\delta)$ be a strong dilation structure. Then
the function introduced in Axiom $(A4)$ yields a product and an
inverse element functions in a neighborhood of the given point.
Precisely, ${\mathcal G}^x=(U(x),x,\operatorname{inv}^x,\Sigma^x)$
$($where $\operatorname{inv}^x,\Sigma^x$ are from Proposition
$\ref{sig_inv}$$)$ is a local group. For the inverse element, the
following property holds:
$\operatorname{inv}^x(\operatorname{inv}^x(u))=u$.
\end{Proposition}

\begin{proof} Let $u,v,w\in
B^d(x,r)$, $\varepsilon\leq\varepsilon_0$, where $r$ is from Axiom
$(A4)$, and $\varepsilon_0$ is such that
$\Sigma_{\varepsilon}^x(u,v)$ is defined for all
$\varepsilon\leq\varepsilon_0$, for example, as in Proposition
\ref{obl_opr}. By direct calculation and using the uniformity of
the limit in (A4) one can verify the following relations:
$$\Sigma^x_{\varepsilon}(x,u)=u;\ \Sigma^{x}_{\varepsilon}(u,\delta^x_{\varepsilon}u)=u;$$
if both parts of the following equality are defined, then
$$\Sigma^x_{\varepsilon}(u,\Sigma^{\delta^x_{\varepsilon}}_{\varepsilon}(v,w))
=\Sigma^x_{\varepsilon}(\Sigma^x_{\varepsilon}(u,v),w);$$
$$\Sigma^x(u,\text{inv}_{\varepsilon}^x(u))=x;\
\Sigma^{\delta^x_{\varepsilon}u}(\text{inv}_{\varepsilon}^x(u),u)=\delta^x_{\varepsilon}u;$$
$$\text{inv}_{\varepsilon}^{\delta^x_{\varepsilon}u}\text{inv}^x_{\varepsilon}u=x.$$
Passing to the limit when $\varepsilon\to 0$, we obtain that
$\Sigma^x(u,v)$ is the product function w. r. t. the identity
element $x$ and inverse function $\text{inv}^x(u)$ such that
$\operatorname{inv}^x(\operatorname{inv}^x(u))=u$. The domains of
the product and inverse functions are some areas $\Omega\supseteq
B^d(x,r)\times B^d(x,r)$, $\Upsilon\supseteq B^d(x,r)$ where $r$
is from (A4). The continuity of functions $\Sigma^x(u,v)$ and
$\operatorname{inv}^xu$ is obvious from (A0), (A4) and Proposition
\ref{sig_inv}.
\end{proof}

\begin{Proposition} \label{dil_avt} The following identities
$$\delta_{\mu}^x\Sigma^x(u,v)=\Sigma^x(\delta_{\mu}^x(u),\delta_{\mu}^x(v)),\
\operatorname{inv}^x(\delta_{\mu}^xu)=\delta_{\mu}^x\operatorname{inv}^xu$$
hold provided both parts of the equality are defined $($in
particular, when $\Sigma^x(u,v)$ exists and $\mu\leq 1$$)$.
\end{Proposition}
\begin{proof}
For the function $$\Lambda^x=\lim\limits_{\varepsilon\to 0}
\Lambda_{\varepsilon}^x(u,v)= \lim\limits_{\varepsilon\to 0}
\delta_{\varepsilon^{-1}}^{\delta_{\varepsilon}^xu}\delta_{\varepsilon}^{x}v$$
from Axiom (A4), direct calculations show that
$\Lambda^x_{\varepsilon}(\delta_{\mu}^xu,\delta_{\mu}^xv)=
\delta_{\mu}^{\delta_{\varepsilon\mu}^xu}\Lambda^x_{\varepsilon\mu}(u,v),$
hence
\begin{equation}\label{del_dil}\delta_{\mu}^x\Lambda^x(u,v)=
\Lambda^x(\delta^x_{\mu}u,\delta_{\mu}^xv),\end{equation} provided
both parts of the last equality are defined. From here the second
equality of the proposition is obvious, since
$\operatorname{inv}^x(u)=\Lambda^x(u,x)$.

The first equality of the proposition follows from
\eqref{del_dil}, \eqref{del_sig} and from the second equality.
\end{proof}

\begin{Proposition}\label{isom} Let $({\mathbb X},d,\delta)$
be a strong nondegenerate uniform dilation structure. Then for all
$u\in B^d(x,r)$ the function $\Sigma^x(u,\cdot)$ $($see
Proposition $\ref{sig_inv}$$)$ is a $d^x$-isometry on $B^d(x,r)$.
\end{Proposition}
\begin{proof}
Using Proposition \ref{cone_prop} and uniformity in Axiom $(A3)$,
we get $$\lim\limits_{\varepsilon\to 0}
 \frac{1}{\varepsilon}\mid
 d(\delta_{\varepsilon}^xv,\delta_{\varepsilon}^xw)-
d^{\delta_{\varepsilon}^xu}(\delta_{\varepsilon}^xv,\delta_{\varepsilon}^xw)\mid=
\lim\limits_{\varepsilon\to 0}
 \mid
 \frac{1}{\varepsilon}d(\delta_{\varepsilon}^xv,\delta_{\varepsilon}^xw)-
d^{\delta_{\varepsilon}^xu}(\delta_{\varepsilon^{-1}}^{\delta_{\varepsilon}^xu}
\delta_{\varepsilon}^xv,\delta_{\varepsilon^{-1}}^{\delta_{\varepsilon}^xu}
\delta_{\varepsilon}^xw)\mid=$$ $$=\mid
d^x(v,w)-d^x(\Lambda^x(u,v),\Lambda^x(u,w))\mid=0,$$ where
$\Lambda^x$ is from Axiom (A4). Further, we have
$$d^x(v,w)=d^x(\Lambda^x(u,\Sigma^x(u,v)),\Lambda^x(u,\Sigma^x(u,w)))=d^x(\Sigma^x(u,v),
\Sigma^x(u,w)).$$ From here the assertion follows.
\end{proof}

It is interesting to compare the following proposition with the
definition and properties of homogeneous norm on a homogeneous Lie
group \cite{fs}.

\begin{Proposition}\label{odnor} Let
$({\mathbb X},d,\delta)$ be a strong nondegenerate dilation
structure. Then the function $|\cdot|:B^d(x,R)\to\mathbb R$,
defined as $|u|=d^x(x,u)$, meets the following properties:

$1)$ homogeneity$:$ if $u\in B^d(x,R)$ and $\delta_r^xu\in
B^d(x,R)$ is defined then $|\delta_r^xu|=r|u|$$;$

$2)$ non-degeneracy: $u=x$ if and only if $|u|=0.$

$3)$ generalized triangle inequality$:$ if for $u,v\in B^d(x,R)$
the value $\Sigma^x(u,v)\in B^d(x,R)$ is defined then the
following inequality holds:
\begin{equation}\label{treug}
|\Sigma^x(u,v)|\leq c\left(|u|+|v|\right), \end{equation} where
$1\leq c<\infty$ and $c=c(x)$ does not depend on $u,v$.
\end{Proposition}
\begin{proof}
The first property directly follows from the conical property; the
second one is equivalent to the assumption of non-degeneracy of
the dilation structure. Let us show 3). Due to continuity  of the
product function $(u,v)\mapsto\Sigma^x(u,v)$ there exists
$0<\tau\leq R$ such that $\bar{B}^{d^x}(x,\tau)\subseteq B^d(x,r)$
and for all $u,v\in \bar{B}^{d^x}(x,\tau)$ we have
$\Sigma^x(u,v)\in B^d(x,R)\cap B^{d^x}(x,R)$. W. l. o. g. assume
$|v|\leq|u|$ and consider first the case when $|u|\leq\tau$ (then
$\varepsilon=\varepsilon(u)=\tau^{-1}|u|\leq 1$).

Let us show that the elements  $\delta^x_{\tau |u|^{-1}}u$, $
\delta^x_{\tau |u|^{-1}}v$ exist and belong to $B^d(x,R)$.

Indeed, it is sufficient to verify that $u\in
W_{\varepsilon^{-1}}(x)$. Since $\varepsilon\tau=|u|$, we have
$u\in \bar{B}^{d^x}(x,\tau\varepsilon)$ (see Remark \ref{balls}).
According to the choice of $\tau$ the following inclusions hold
$\bar{B}^{d^x}(x,\tau)\subseteq B^d(x,r)\subseteq B^d(x,R)$,
therefore, due to axiom (A0) and Proposition \ref{cone_prop}, it
is true that $u\in
\bar{B}^{d^x}(x,\tau\varepsilon)=\delta_{\varepsilon}^x\bar{B}^{d^x}(x,\tau)\subseteq
\delta_{\varepsilon}^xB^d(x,R)\subseteq
V_{\varepsilon}(x)\subseteq W_{\varepsilon^{-1}}(x)$. Note that it
can not happen that $\delta_{\varepsilon^{-1}}^xu\in U(x)\setminus
B^d(x,R)$, because
$\delta_{\varepsilon^{-1}}^xB^d(x,R\varepsilon)\subseteq
\delta_{\varepsilon^{-1}}^x\delta_{\varepsilon}^x B^d(x,R)=
B^d(x,R)$.

Thus, due to 1), $|\delta^x_{\tau |u|^{-1}}u|=d^x(x,
\delta^x_{\tau |u|^{-1}}u)=\tau,\ |\delta^x_{\tau
|u|^{-1}}v|\leq\tau$. Hence, by choice of $\tau$, the value
$\Sigma^x(\delta^x_{\tau |u|^{-1}}u, \delta^x_{\tau |u|^{-1}}v)\in
B^d(x,R)\cap B^{d^x}(x,R)$ is defined. Thus, from Proposition
\ref{dil_avt}, we can derive
 $$
 \Sigma^x(u,v)
=\delta^x_{\tau^{-1}|u|}\Sigma^x(\delta^x_{\tau |u|^{-1}}u,
\delta^x_{\tau |u|^{-1}}v).
 $$
It follows immediately that
\begin{multline*}
|\Sigma^x(u,v)|=|\delta^x_{\tau^{-1}|u|} ( \Sigma^x(\delta^x_{\tau
|u|^{-1}}u, \delta^x_{\tau |u|^{-1}}v))|\\
=\tau^{-1}|u| | \Sigma^x(\delta^x_{\tau |u|^{-1}}u, \delta^x_{\tau
|u|^{-1}}v) |\leq c|u|\leq c(|u|+|v|),
\end{multline*}
where $c=\tau^{-1}R$.

Let now be $|u|>\tau$ and $\Sigma^x(u,v)\in B^d(x,R)$ be defined.
Choose $0<\mu<1$ such that $\delta_{\mu}^xu, \delta_{\mu}^xu\in
B^{d^x}(x,\tau)$ (such $\mu$ exists due to continuity of
dilations). Then
$$\mu|\Sigma^x(u,v)|=|\delta_{\mu}^x\Sigma^x(u,v)|=|\Sigma^x(\delta_{\mu}^xu,\delta_{\mu}^xv)|
\leq c(|\delta_{\mu}^xu|+|\delta_{\mu}^xv|)=c\mu(|u|+|v|).$$ It
follows \eqref{treug}.
\end{proof}

\begin{Definition} The function $|\cdot|$,
introduced in Proposition \ref{odnor}, is said to be
 {\it the homogeneous
norm} on the local group ${\mathcal G}^x$.\end{Definition}

\begin{Definition}[\rm\cite{mal}]\label{glob_as_def}
It is said that for a local group ${\mathcal G}$ the {\it global
associativity property} holds if there is a neighborhood of the
identity  $V\subseteq{\mathcal G}$ such that for each $n$-tuple of
elements $a_1,a_2\ldots,a_n\in V$ whenever there exist two
different ways of introducing parentheses in this $n$-tuple, so
that all intermediate products are defined, the resulting products
are equal.
\end{Definition}

\begin{Theorem}[\rm Mal'tsev \cite{mal}]\label{mal}
A local topological group ${\mathcal G}$  is locally isomorphic to
a some topological group $G$  if and only if the global
associativity property in ${\mathcal G}$ holds.
\end{Theorem}

\begin{Remark}
Unlike in the case of global groups, the verification of the
global associativity property for local groups is a nontrivial
task. This verification can not be done by a trivial induction as
for global groups since it would require to assume the existence
of all intermediate products which is, in general, not true for
local groups. See comments in  \cite{olv,gol} where there are some
references to papers with mistakes caused by misunderstandings of
this fact. In the local group ${\mathcal G}^x$ under our
consideration it is easy to provide  examples for $n=4$ such that
$u_i\in B^d(x,R)$ and combinations
$u=\Sigma^x(\Sigma^x(u_1,\Sigma^x(u_2,u_3)),u_4)$ and
$u^{\prime}=\Sigma^x(u_1,\Sigma^x(u_2,\Sigma^x(u_3,u_4)))$ exist
while the combination
$\Sigma^x(\Sigma^x(u_1,u_2),\Sigma^x(u_3,u_4)))$ is not defined.
More examples can be found in \cite{mal,olv}.
\end{Remark}

\begin{Proposition}\label{glob_as}
For the local group ${\mathcal G}^x$, the global associativity
property holds.
\end{Proposition}
\begin{proof}
Let $u_1,u_2,\ldots,u_n\in B^d(x,R)$, and $u,u^{\prime}$ be
elements obtained from the $n$-tuple $(u_1,u_2,\ldots,u_n)$ by
introducing parentheses such that the products exist. We need to
show that $u=u^{\prime}$.

Let $\tau$ be such as in the proof of Proposition \ref{odnor},
$R_x=\inf\{\xi\mid B^d(x,R)\subseteq B^{d^x}(x,\xi)\}$,
$c_n=nc^{n-1}$ where $c$ is from \eqref{treug}. Let
$s_n=\frac{\tau}{c_{n-1}R_x}$ and $\tilde{u_i}=\delta_{s_n}^xu_i$.
By induction on $n$ and using  \eqref{treug} it is easy to show
that all possible products of length not bigger than $n$ of the
elements $\tilde{u_i}$ are defined. Thus it can be trivially shown
(as for global groups) that $\delta_{s_n}^x(u)=
\delta_{s_n}^x(u^{\prime})$. Applying to both sides of the last
equality the homeomorphism
 $\delta^x_{s_n^{-1}}$ (which is, in particular, an injective mapping),
 we get $u=u^{\prime}$ and finish the proof.
\end{proof}

\begin{Definition}[\rm\cite{sieb}, Proposition 5.4]\label{contract}
A topological group $G$ is {\it contractible} if there is an
automorphism  $\tau: G\to G$ such that
$\lim\limits_{n\to\infty}\tau^ng=e$ for all $g\in G$.

\end{Definition}

\begin{Definition}
A topological space is {\it locally compact} if any of its points
has a neighborhood the closure of which is compact. A local group
is {\it locally compact} if there is a neighborhood of its
identity element the closure of which is compact.
\end{Definition}

The proof of Theorem \ref{dil_teor} relies on the following
statement, see Remark \ref{bul_rem} for comments.

\begin{Proposition}[\rm\cite{sieb}, Corollary 2.4]\label{sieb} For a connected
locally compact group $G$ the following assertions are equivalent:

$(1)$ $G$ is contractible;

$(2)$ $G$ is a simply connected Lie group the Lie algebra $V$ of
which is nilpotent and  graded, i.~e. there is a decomposition
$V=\bigoplus\limits_{s>0}V_s$ such that $[V_s,V_t]\subseteq
V_{s+t}$ for all $s,t> 0.$ In particular,  $V$ is
nilpotent.\end{Proposition}

\begin{Theorem}\label{dil_teor} Let $({\mathbb X},d,\delta)$ be
a strong nondegenerate dilation structure. Then

$1)$ For any $x\in {\mathbb X}$, the local group ${\mathcal G}^x$
is locally isomorphic to a connected simply connected Lie group
$G^x$ the Lie algebra of which is nilpotent and graded;

$2)$ If the dilation structure is, in addition, uniform, then the
Lie group $G^x$ is the tangent cone $($in the sense of Definition
$\ref{cone_def}$$)$ to ${\mathbb X}$ at $x$, i.~e., left
translations on $G^x$ are isometries w. r. t. quasimetric
$\tilde{d}^x$ on $G^x$ which arises from $d^x$ in a natural way.
The local group ${\mathcal G}^x$ is a local tangent cone.
\end{Theorem}
\begin{proof}
Since ${\mathbb X}$ is boundedly compact,  ${\mathcal G}^x$ is a
locally compact local group. Due to existence on ${\mathcal G}^x$
of a one-parameter family of dilations this local group is
linearly connected (indeed, any two points $u,v\in U(x)$ can be
connected by the continuous curve
$\{\delta^x_{\varepsilon}(u)\}_{1\geq \varepsilon\geq
0}\circ\{\delta^x_{\varepsilon}(v)\}_{0\leq \varepsilon\leq 1}$),
hence ${\mathcal G}^x$ is connected.

 According to Proposition
\ref{glob_as}, the global associativity property in ${\mathcal
G}^x$ holds. Hence, by Theorem \ref{mal}, ${\mathcal G}^x$ is
locally isomorphic to some topological group $G^x$. Let us use the
construction of this group given in the proof of Theorem \ref{mal}
in \cite{mal} and in more details in \cite{gol_dries}: $G^x$ is
obtained as the group of equivalence classes of words arranged
from elements of the initial local group ${\mathcal G}^x$.

Namely, let ${\mathcal G}^x_{(n)}=\{(u_1,\ldots,u_n)\mid
u_i\in{\mathcal G}^x\}$ be the set of words of length $n$, and
$\tilde{G}^x=\bigcup\limits_{n\in\mathbb N}{\mathcal G}^x_{(n)}$.
On $\tilde{G}^x$ the following two operations can be introduced.
The contraction is defined as
$$(u_1,\ldots,u_n)\in{\mathcal G}^x_{(n)}
\mapsto (u_1,\ldots,u_{i-1},\Sigma^x(u_i,u_{i+1}),u_{i+2},\ldots,
u_n)\in{\mathcal G}^x_{(n-1)},$$ if $\Sigma^x(u_i,u_{i+1})$
exists. The expansion is defined as
$$(u_1,\ldots,u_n)\in{\mathcal G}^x_{(n)}\mapsto (u_1,\ldots,u_{i-1},v,w,u_{i+1},
\ldots, u_n)\in{\mathcal G}^x_{(n+1)},$$ if $u_i=\Sigma^x(v,w)$.
Two words $(u_1,\ldots, u_n)$ and $(v_1,\ldots,v_m)$ are called
equivalent (which is denoted as $(u_1,\ldots,
u_n)\sim(v_1,\ldots,v_m)$) if they can be obtained one from
another by a finite sequence of contractions and expansions.
Finally, let $G^x=\tilde{G}^x/\sim$.
 The product and inverse functions and the neutral element on  $G^x$ are
 defined respectively as
 $$[(u_1,\ldots,u_n)]\cdot[(v_1,\ldots,v_m)]=[(u_1,\ldots,u_n,v_1,\ldots,v_m)],$$
  $$[(u_1,\ldots,u_n)]^{-1}=[(\text{inv}^xu_{n},\ldots,\text{inv}^xu_{1})],
  \ e_{G^x}=[(e_{{\mathcal G}^x})].$$
  It is easy to verify that the function $\varphi:{\mathcal G}^x\to G^x$
  which maps the element
   $g$  to the equivalence class $[(g)]$, is a local isomorphism.

  The topology on $G^x$ is defined as follows: if ${\mathcal B}$ is the
  basis of topology
  of ${\mathcal G}^x$, then $B=\{\varphi(U)\mid U\in{\mathcal B}\}$
  is the base of topology
  of $G^x$. The  verification
  of axioms of a topological basis can be done straightforwardly.

For an arbitrary $s<1$ define a contractive automorphism on $G^x$
as
$$\tau([(u_1,\ldots,u_n)])=
[(\delta^x_s(u_1),\ldots,\delta^x_s(u_n))].$$ Due to the linear
connectedness of the group $G^x$  (because of the obvious relation
$[(e,e,$ $\ldots,e)]$$=[(e)]$ and the fact that the local group
 ${\mathcal G}^x$ is linearly connected), by Proposition
 \ref{sieb} we get the first assertion of the theorem.

 Now let $s_{mn}=s_{\max\{m,n\}}$ (in notation of the proof of Proposition
  \ref{glob_as})
 and define on $G^x$ a quasimetric as
\begin{multline*}
\tilde{d}^x([(u_1,\ldots,u_n)],[(v_1,\ldots,v_m)])
\\=\frac{1}{s_{mn}}d^x(\Sigma^x(\delta_{s_{mn}}^xu_1,\ldots,\delta_{s_{mn}}^xu_n),
\Sigma^x(\delta_{s_{mn}}^xv_1, \ldots,\delta_{s_{mn}}^xv_m)).
\end{multline*}
Note that  Propositions \ref{cone_prop}, \ref{dil_avt} imply the
generalized triangle inequality for $\tilde{d}^x$ with the
constant $Q_{\mathbb X}$ and that $\varphi$ is an isometry. Taking
into account Theorem \ref{exist_teor} and Proposition \ref{isom}
we obtain the second assertion.
\end{proof}

\begin{Remark}
 Let us
give a brief overview of the proof of Proposition \ref{sieb}, for
showing that it can not be straightforwardly applied to the case
of local groups. The crucial part of this proof is to show that a
connected locally compact contractible group is a Lie group. This
proof heavily relies on several main theorems from the book of
Montgomery and Zippin \cite{mz}, where the solution of H5 is
given. The proofs of those theorems are long and complicated, and,
as noted in \cite[p.~119]{mz}, ``Most of the Lemmas can be also
proved by essentially the same arguments for the case of a locally
compact connected {\it local} group but we shall not complicate
the statements and proofs of the Lemmas by inserting the necessary
qualifications.'' This last statement shows, that proving the
theorems (based on the mentioned lemmas) that we would need, for
the case of local groups, is, at least, nontrivial (and not done
by anybody, as far as we know). It
 would require a
careful study of large parts of the book \cite{mz}.

Overcoming this difficulty we apply Mal'tsev's theorem \ref{mal}
to reduce the consideration to the case of (global) groups, for
which Proposition \ref{sieb} can be applied.
\end{Remark}

\begin{Remark}\label{u_n}
There is an another look at the proof of Proposition
\ref{glob_as}. It
 actually can be proved without the triangle
inequality \eqref{treug} and any (quasi)metric structure, by means
of the following simple topological fact (\cite[Chapter 3, Section
23, E]{pon},
 see also \cite{gol}): in any local group
there is a decreasing sequence of neighborhoods $\{{\mathcal
U}_n\}_{n\in\mathbb N}$ such that, for all elements $u_1,\ldots
u_n\in {\mathcal U}_n$, their products are defined with any
combinations of parentheses. Using this fact, an analog of Theorem
\ref{dil_teor}, for locally compact topological spaces with
dilations, can be proved (for this purpose, axioms of Definition
\ref{dil_struct} should be modified in a natural way).
Globalizability of locally compact locally connected contractible
local groups was proved in \cite{gol_dries}, independently of our
paper. The result of \cite{gol_dries} can be viewed as a
generalization of the first assertion of Theorem \ref{dil_teor}.

On the other hand, using the (quasi)metric structure allows to
make the proof of global associativity more constructive in
comparison with the purely topological one.
\end{Remark}

\section{Example: Carnot-Carath\'{e}odory spaces}\label{cc_sec}

\begin{Definition}[\cite{bel,gro,karm, nsw, vk, vk1, vk2}]
\label{cc} Fix a connected Riemannian
$C^{\infty}$-mani\-fold~$\mathbb M$ of dimension~$N$. The
manifold~$\mathbb M$ is called a {\it regular
Carnot-Carath\'{e}odory space} if in the tangent bundle $T\mathbb
M$ there is a filtration
$$H\mathbb M=H_1\mathbb M\subseteq\ldots\subseteq H_i\mathbb M\subseteq\ldots\subseteq
H_M\mathbb M=T\mathbb M$$ of subbundles of the tangent bundle
$T\mathbb M$, such that, for each point $p\in\mathbb M$, there
exists a neighborhood $U\subset\mathbb M$
 with a collection of $C^{1,\alpha}$ (where $\alpha\in (0,1]$) vector fields
$X_1,\dots,X_N$ on $U$ enjoying the following three properties.
For each $v\in U$ we have

$(1)$ $X_1(v),\dots,X_N(v)$ constitutes a basis of $T_v\mathbb M$;

$(2)$ $H_i(v)=\operatorname{span}\{X_1(v),\dots,X_{\dim H_i}(v)\}$
is a subspace of $T_v\mathbb M$ of  dimension $\dim H_i$,
$i=1,\ldots,M$, where $H_1(v)=H_v\mathbb M$;

$(3)$
\begin{equation}\label{tcomm}[X_i,X_j](v)=\sum\limits_{\operatorname{deg}
X_k\leq \operatorname{deg} X_i+\operatorname{deg}
X_j}c_{ijk}(v)X_k(v)
\end{equation}
where the {\it degree} $\deg X_k$ equals $\min\{m\mid X_k\in
H_m\}$;

The number $M$ is called the {\it depth} of the
Carnot-Carath\'{e}odory space.
\end{Definition}

\begin{Remark} According to \cite{vk}, all statements below are also valid for the case
when $X_i\in C^1$ and $M=2$. \end{Remark}

\begin{Definition}\label{exp}
For any point $g\in\mathbb M$, define the mapping
\begin{equation}\label{exp1}
\theta_g(v_1,\ldots,v_N)=\exp\biggl(\sum\limits_{i=1}^Nv_iX_i\biggr)(g).
\end{equation}
It is known that $\theta_g$ is a $C^1$-diffeomorphism of the
Euclidean ball  $B_E(0,r)\subseteq \mathbb R^N$ to $\mathbb M$,
where $0\leq r<r_g$ for some (small enough) $r_g$. The collection
$\{v_i\}_{i=1}^N$ is called {\it the normal coordinates} or {\it
the coordinates of the $1^{\text{st}}$ kind  $($with respect to
$u\in\mathbb M)$} of the point $v\in U_g=\theta_g(B_E(0,r_g))$.
Further we will consider a compactly embedded neighborhood
 ${\mathcal
U}\subseteq\mathbb M$ such that ${\mathcal
U}\subseteq\bigcap\limits_{g\in{\mathcal U}}U_g$.
\end{Definition}

\begin{Definition}\label{d} By means of coordinates \eqref{exp1},
 introduce on ${\mathcal U}$
the following quasimetric $d_{\infty}$. For $u,v\in {\mathcal U}$
such that $v=\exp\Bigl(\sum\limits_{i=1}^Nv_iX_i\Bigr)(u)$ let
$$d_{\infty}(u,v)=\max\limits_i\{|v_i|^{\frac{1}{\deg X_i}}\}. $$
\end{Definition}

The properties (1), (2) of Definition \ref{km} for the function
$d_{\infty}$ and its continuity on both arguments obviously follow
from properties of the exponential mapping. The generalized
triangle inequality is proved in \cite{karm,vk}.
 We denote the balls w. r. t. $d_{\infty}$ as
$\operatorname{Box}(u,r)=\{v\in {\mathcal U} \mid
d_{\infty}(v,u)<r\}.$

\begin{Definition}\label{dil_cc}
Define in ${\mathcal U}$ the action of the dilation group
$\Delta^g_{\varepsilon}$ as follows: it maps an element
$x=\exp\Bigl(\sum\limits_{i=1}^Nx_iX_i\Bigr)(g)\in {\mathcal U}$
to the element
$$\Delta^g_{\varepsilon}x=\exp\Bigl(\sum\limits_{i=1}^Nx_i\varepsilon^{\deg
X_i}X_i\Bigr)(g)\in {\mathcal U}$$ in the case when the right-hand
part of the last expression makes sense.\end{Definition}

\begin{Proposition}[\rm\cite{vk}]\label{nilpot}
The coefficients
$$
\bar{c}_{ijk}=
\begin{cases} c_{ijk}(g)\text{ of \eqref{tcomm} },& \text{if }\operatorname{deg}
X_i+\operatorname{deg}X_j=\operatorname{deg}X_k \\
0, &\text{in other cases}
\end{cases}
$$
define a graded nilpotent Lie algebra.

This Lie algebra  can be represented by  vector fields
$\{(\widehat{X}_i^g\}_{i=1}^N\in C^{\alpha}$ on ${\mathcal U}$
such that
\begin{equation}\label{tcommnilp}
[\widehat{X}_i^g,\widehat{X}_j^g]=\sum\limits_{\operatorname{deg}
X_k=\operatorname{deg} X_i+\operatorname{deg}
X_j}c_{ijk}(g)\widehat{X}_k^g
\end{equation}
and $\widehat{X}_i^g(g)=X_i(g)$.
\end{Proposition}

\begin{Definition}\label{loc_carno} To the Lie algebra
$\{\widehat{X}_i^g\}_{i=1}^N$ there corresponds the Lie group
${\mathcal G}^g=({\mathcal U},g,^{-1},*)$ at $g$. The product
function $*$ is defined as follows: if
$x=\exp\Bigl(\sum\limits_{i=1}^Nx_i\widehat{X}^g_i\Bigr)(g)$,
$y=\exp\Bigl(\sum\limits_{i=1}^Ny_i\widehat{X}^g_i\Bigr)(g)$, then
$x* y=\exp\Bigl(\sum\limits_{i=1}^Ny_i\widehat{X}^g_i\Bigr)\circ
\exp\Bigl(\sum\limits_{i=1}^Nx_i\widehat{X}^g_i\Bigr)(g)
=\exp\Bigl(\sum\limits_{i=1}^Nz_i\widehat{X}^g_i\Bigr)(g)$, where
$z_i$ are computed via Campbell-Hausdorff formula. The inverse
element to
$x=\exp\Bigl(\sum\limits_{i=1}^Nx_i\widehat{X}^g_i\Bigr)(g)$ is
defined as
$x^{-1}=\exp\Bigl(\sum\limits_{i=1}^N(-x_i)\widehat{X}^g_i\Bigr)(g)$.\end{Definition}

\begin{Remark}\label{usl4} In the ``classical''
sub-Riemannian setting (see Introduction), the local Lie group
from Definition
 \ref{loc_carno} is locally isomorphic to a
 {\it Carnot group}, i.e., a connected simply connected Lie group
 the Lie algebra
$V$ of which can be decomposed into a direct sum
$V=V_1\oplus\ldots\oplus V_M$ such that $[V_1,V_i]= V_{i+1},\
i=1,\ldots M-1$, $[V_1,V_M]=\{0\}$. In the case under our
consideration, for the Lie algebra of the local group ${\mathcal
G}^g$ only the inclusion $[V_1,V_i]\subseteq V_{i+1}$ is true. The
converse inclusion will hold if we require  an additional
condition \cite{karm,vk} in Definition \ref{cc}: the quotient
mapping $[\,\cdot ,\cdot\, ]_0:H_1\times H_j/H_{j-1}\mapsto
H_{j+1}/H_{j}$ induced by Lie brackets is an epimorphism for all
$1\leq j<M$. Under this additional assumption, an analog of the
Rashevskii-Chow theorem can be proved.
\end{Remark}

Strictly speaking, the group operation is defined on a
neighborhood defined by vector fields $\{\widehat{X}^g_i\}$, but,
w. l. o. g., we can assume that this neighborhood coincides with
${\mathcal U}$ \cite{vk, vk2}. Note also that the mapping
$\theta_g$ is a local isometric isomorphism between the local Lie
group
 $({\mathcal G}^g,*)$ and the Lie group
$(\mathbb R^N,*)$, and $\theta_g(0)=g$. The group operation $*$ on
$\mathbb R^N$ is introduced by analogy with Definition
\ref{loc_carno}, by means of $C^{\infty}$ vector fields
$\{(\widehat{X}^g_i)^{\prime}\}$ on $\mathbb R^N$, such that
$\widehat{X}_i^g=(\theta_g)_{*}(\widehat{X}_i^g)^{\prime}$, where
$(\theta_g)_{*}\langle {Y}\rangle$$(\theta_g(x))=D
\theta_g(x)\langle {Y}(x)\rangle$, ${Y}\in T\mathbb R^N$ (see
details in \cite{karm,vk,vk2}). In what follows, we will identify
the neighborhood ${\mathcal U}$ with its image
$\theta_g^{-1}({\mathcal U})\subseteq\mathbb R^N$.

This identification allows, in particular, to define canonical
coordinates of the first kind, induced by the nilpotentized vector
fields in a similar way as \ref{exp}.

\begin{Definition}\label{d_g} For $u,v\in \mathbb R^N$ such that
$v=\exp\Bigl(\sum\limits_{i=1}^Nv_i(\widehat{X}^g_i)^{\prime}\Bigr)(u)$,
let $d_{\infty}^g(u,v)=\max\limits_i\{|v_i|^{\frac{1}{\deg
X_i}}\}.$\end{Definition}

It is known \cite{fs}  that $d_{\infty}^g$ is a quasimetric. We
denote the balls w. r. t. this quasimetric as
$\operatorname{Box}^g(u,r)=\{v\in \mathbb R^N\mid
d_{\infty}^g(v,u)<r\}.$

\begin{Proposition}[\rm \cite{vk,v}]\label{box} If $r$ is such that
$\operatorname{Box}(g,r)\subseteq{\mathcal U}$ then
$\operatorname{Box}(g,r)=\operatorname{Box}^g(g,r)$.\end{Proposition}

\begin{Definition}\label{dil_cc_nilp}
The nilpotentized vector fields also define dilations on
${\mathcal U}$: the element
$x=\exp\Bigl(\sum\limits_{i=1}^Nx_i\widehat{X}^g_i\Bigr)(g)\in
{\mathcal U}$ is mapped to the element
$$\delta^g_{g,\varepsilon}x=\exp\Bigl(\sum\limits_{i=1}^Nx_i\varepsilon^{\deg
X_i}\widehat{X}^g_i\Bigr)(g)\in {\mathcal U}$$ in the case when
the right-hand part of the last expression makes sense.
\end{Definition}

\begin{Proposition}[\rm \cite{vk,v}]\label{del_eq}
For all $\varepsilon>0$ and $u\in {\mathcal U}$, we have
$\Delta^g_{\varepsilon}u=\delta^g_{g,\varepsilon}u$, if both parts
of this equality are defined.
\end{Proposition}

\begin{Proposition} [\rm\cite{fs,vk,vk2}]
\label{cone_prop_cc} The cone property for the quasimetric
$d_{\infty}^g(u,v)$ holds:
 $d_{\infty}^g(u,v)= \frac{1}{\varepsilon}
d_{\infty}^g(\Delta^g_{\varepsilon}u, \Delta^g_{\varepsilon}v)$
for all possible $\varepsilon>0$.
\end{Proposition}

\begin{Theorem}[\rm Estimate on divergence of integral
lines \cite{karm,vk}]\label{int_lines} Consider points $u,v\in
\mathcal U$ and
$$
w_{\varepsilon}=\exp\Bigl(\sum\limits_{i=1}^Nw_i\varepsilon^{\operatorname{deg}
X_i}X_i\Bigr)(v) \text{ and }
\widehat{w}_{\varepsilon}=\exp\Bigl(\sum\limits_{i=1}^Nw_i\varepsilon^{\operatorname{deg}
X_i}\widehat{X}^{u}_i\Bigr)(v).
$$
Then
\begin{equation}\label{locgeomest}
\max\{d_{\infty}^u(w_{\varepsilon},\widehat{w}_{\varepsilon}),
d_{\infty}^{u^{\prime}}(w_{\varepsilon},\widehat{w}_{\varepsilon})\}=\varepsilon
[\Theta(u,v,\alpha,M)]\rho(u,v)^{\frac{\alpha}{M}},
\end{equation}
where $\Theta$ is uniformly bounded on $u,v\in{\mathcal U}.$

\end{Theorem}

\begin{Theorem} [\rm Local approximation theorem
\cite{bel,gre,gro,karm,vk,vk1}] \label{lat} If $u,v\in
\operatorname{Box}(g,\varepsilon)$, then
$\left|d_{\infty}(u,v)-d_{\infty}^g(u,v)\right|=O(\varepsilon^{1+\frac{\alpha}{M}})$
uniformly on $g\in {\mathcal U},\ u,v\in
\operatorname{Box}(g,\varepsilon)$.
\end{Theorem}

\begin{Theorem}\label{cc_teor} Dilations from  Definition $\ref{dil_cc}$
induce on the quasimetric space $({\mathcal U},d_{\infty})$ a
strong uniform nondegenerate dilation structure with the conical
quasimetric $($$d^x$ from Axiom $(A3)$$)$
$d^g_{\infty}$.\end{Theorem}

\begin{proof} Axioms (A0)~--- (A2) and non-degeneracy of Definition \ref{dil_struct}
obviously hold due to properties of exponential mappings; (A3) and
uniformity directly follow from Theorem \ref{lat}.

Axiom (A4) follows from group operation properties and Theorem
\ref{int_lines}. Indeed, let
$u=\exp\Bigl(\sum\limits_{i=1}^Nu_iX_i\Bigr)(g),\ $
$v=\exp\Bigl(\sum\limits_{i=1}^Nv_iX_i\Bigr)(g)\in {\mathcal U}$.
We need to show the existence and uniformity of the limits of
$\Sigma_{\varepsilon}^g(u,v)=
\Delta_{\varepsilon^{-1}}^g\Delta_{\varepsilon}^{\Delta_{\varepsilon}^gu}v$
and
$\operatorname{inv}^g_{\varepsilon}(u)=\Delta_{\varepsilon^{-1}}^{\Delta_{\varepsilon}^gu}g$,
when
 $\varepsilon\to 0$ (see Proposition \ref{sig_inv}).

First we prove the existence of the limit on the local group (i.
e. replacing $\Delta^g_{\varepsilon}$ by
$\delta^g_{g,\varepsilon}$) According to (A2),
$\lim\limits_{\varepsilon\to
0}\Delta_{\varepsilon}^xu=\lim\limits_{\varepsilon\to
0}u_{\varepsilon}=g$. By means of \eqref{exp} we can write
$$v=\exp\Bigl(\sum\limits_{i=1}^N\tilde{v}^{\varepsilon}_iX_i\Bigr)
(u_{\varepsilon}).$$ Since the coordinates of the first kind are
uniquely defined,
 \begin{equation}\label{predel}\lim\limits_{\varepsilon\to
0}\tilde{v}^{\varepsilon}_i=v_i,\ i=1,\ldots,N.\end{equation}
  Now let
$$a=\delta^{u_{\varepsilon}}_{g,\varepsilon}v=\exp\Bigl(\sum\limits_{i=1}^N\tilde{v}^
{\varepsilon}_i \varepsilon^{\operatorname{deg}
X^g_i}\widehat{X}_i\Bigr)\circ \exp\Bigl(\sum\limits_{i=1}^N
u_i\varepsilon^{\operatorname{deg} X_i}\widehat{X}^g_i\Bigr)(g).$$
Then
$$\Sigma_{\varepsilon}^g(u,v)=\delta_{g,\varepsilon^{-1}}^ga=
\exp\Bigl(\sum\limits_{i=1}^N\tilde{v}^{\varepsilon}_i
(\delta^g_{g,\varepsilon^{-1}})_*(\varepsilon^{\operatorname{deg}
X_i}\widehat{X}^g_i)\Bigr)\circ \exp\Bigl(\sum\limits_{i=1}^N
u_i(\delta^g_{g,\varepsilon^{-1}})_*(
\varepsilon^{\operatorname{deg} X_i}\widehat{X}^g_i)\Bigr)(g).$$
Using group homogeneity and \eqref{predel}, we get the existence
of the uniform (on $g$) limit
$$\lim\limits_{\varepsilon\to
0}\Sigma_{\varepsilon}^g(u,v)=\exp\Bigl(\sum\limits_{i=1}^Nv_i\widehat{X}^g_i\Bigr)\circ
\exp\Bigl(\sum\limits_{i=1}^Nu_i\widehat{X}^g_i\Bigr)(g).$$

Now let us estimate the difference between the two combinations.
From Properties \ref{del_eq}, \ref{cone_prop_cc} and Theorem
\ref{int_lines} we infer
$$d_{\infty}^g\left(\Delta_{\varepsilon^{-1}}^g
\Delta_{\varepsilon}^{\Delta_{\varepsilon}^gu}v,\delta_{g,\varepsilon^{-1}}^g
\delta_{g,\varepsilon}^{\delta_{g,\varepsilon}^gu}v\right)=
d_{\infty}^g\left(\Delta_{\varepsilon^{-1}}^g
\Delta_{\varepsilon}^{\Delta_{\varepsilon}^gu}v,\Delta_{\varepsilon^{-1}}^g
\delta_{g,\varepsilon}^{\Delta_{\varepsilon}^gu}v\right)=$$
$$=\varepsilon^{-1}d_{\infty}^g\left(\Delta_{\varepsilon}^{u_{\varepsilon}}v,
\delta_{g,\varepsilon}^{u_{\varepsilon}}v\right)=\varepsilon^{-1}\cdot
O\left(\varepsilon^{1+\frac{1}{\alpha}}\right)\to 0$$ when
$\varepsilon\to 0$, which implies the uniform convergence of
$\Sigma_{\varepsilon}^g(u,v)$.

Concerning the inverse element, we have
$$u_{\varepsilon}=\exp\Bigl(\sum\limits_{i=1}^Nu_i\varepsilon^{\deg X_i}X_i\Bigr)
(g),\ g=\exp\Bigl(\sum\limits_{i=1}^N-u_i\varepsilon^{\deg
X_i}X_i\Bigr) (u_{\varepsilon}),$$ hence
$$\text{inv}^g(u,v)=\text{inv}_{\varepsilon}^g(u,v)=\exp\Bigl(\sum\limits_{i=1}^N
-u_iX_i\Bigr)(g),$$
 which finishes the proof.

\end{proof}

\begin{Remark}
In contrast to the proof of a similar assertion in \cite{bul_sr1},
 we do not use, for proving Theorem \ref{cc_teor},
 the normal frames technique \cite{bel}.

 Nevertheless, our considerations include, as a particular case, the ``classical''
 sub-Riemannian setting, although in this setting the number of nontrivial commutators
 of ``horizontal'' vector fields can be bigger then the dimension $N$ of the manifold
 $\mathbb M$. Indeed, the nilpotent Lie algebras, defined by different bases, are
 isomorphic to each other
 due to the functorial property of the tangent cone \cite{v,vk}. Analogs of the basic
 Theorems  \ref{lat}, \ref{int_lines}, needed for the proof of Theorem \ref{cc_teor}
  for the intrinsic metric $d_c$ are proved in
 \cite{bel,vk,vk1}.\end{Remark}

\begin{Remark} An analog of Theorem \ref{cc_teor} can be proved
for some other quasimetrics equivalent to $d_{\infty}$,  looking
like e. g. in \cite{bel}.

Note also that proofs in \cite{karm} do not use tools concerned
with the Baker-Campbell-Hausdorff formula.

\end{Remark}

\section{Differentiability}
 Let $(\mathbb X,d_{\mathbb
X},\delta)$ and $(\mathbb Y,d_{\mathbb Y},\tilde{\delta})$ be two
quasimetric spaces with strong nondegenerate dilation structures.
In this section we denote the local group $\mathcal G^x$  at
$x\in\mathbb X$ ($\mathcal G^y$ at $y\in\mathbb Y$ ) by the symbol
$\mathcal G^x\mathbb X$ ($\mathcal G^y\mathbb Y$). Quasimetrics on
them will be denoted by $d^x$ and $ d^y$ respectively.

Recall that a {\it $\delta$-homogeneous homomorphism} of graded
nilpotent  groups $\mathbb G$ and $\widetilde{\mathbb G}$ with
one-parameter groups of dilations $\delta$ and $\tilde\delta$
\cite{fs} respectively is a continuous homomorphism $L:\mathbb
G\to\widetilde{\mathbb G}$ of these groups such that

 $$L\circ \delta=\tilde\delta \circ L.$$

The case of local  graded nilpotent groups $\mathcal G$ and
$\widetilde{\mathcal G}$ with one-parameter groups of dilations
$\delta$ and $\tilde\delta$ respectively
  is different from this
only in that the equality $L\circ \delta(v)=\tilde\delta \circ
L(v)$
 holds only for $v\in{\mathcal G}$
  and $t>0$
such that  $\delta_t v\in{\mathcal G}$ and $\tilde \delta_t
L(v)\in \widetilde{\mathcal G}$.

\begin{Definition}\label{defdiff}
Given two quasimetric spaces $(\mathbb X,d_{\mathbb X},\delta)$
and $(\mathbb Y,d_{\mathbb Y},\tilde\delta)$ with strong uniform
 nondegenerate dilation structures, and a~set $E\subset\mathbb X$.
A mapping $f:E\to {\mathbb Y}$ is called {\it
$\delta$-differentiable} at a point $g\in E$ if there exists a
$\delta$-homogeneous homomorphism
 $L:\bigl(\mathcal G^g\mathbb X, d^g\bigr)\to
 \bigl(\mathcal G^{f(g)}\mathbb Y,   d^{f(g)}\bigr)$
of the local nilpotent tangent cones such that
\begin{equation}\label{estdiff}
   d^{f(g)}(f(v),L(v))=o\bigl(d^g(g,v)\bigr)
  \quad\text{as  $E\cap {\mathcal G^g\mathbb X}\ni v\to g$}.
\end{equation}
\end{Definition}

A  $\delta$-homogeneous homomorphism $L:\bigl({\mathcal
G}^g\mathbb X,d^g\bigr)\to \bigl({\mathcal G}^{f(g)}\mathbb
Y,{d}^{f(g)}\bigr)$ satisfying condition \eqref{estdiff}, is
called a $\delta$-{\it differential\/} of the mapping
$f:E\to{\mathbb Y}$ at $g\in E$ on $E$ and is denoted by  $Df(g)$.
 It can be proved like in \cite{v,v1} that if ${E=\mathbb X}$
 then the $\delta$-differential is unique.

Moreover, it is easy to verify that a homomorphism
$L:\bigl({\mathcal G}^g\mathbb X,d^g\bigr)\to \bigl({\mathcal
G}^{f(g)}\mathbb Y,{d}^{f(g)}\bigr)$ satisfying \eqref{estdiff}
commutes with the one-parameter dilation group:
\begin{equation}\label{commut}
\tilde \delta^{f(g)}_t\circ L=L\circ \delta^g_t,
\end{equation}
i.e., $L$ is $\delta$-homogeneous homomorphism.

In the case of Carnot groups, the introduced concept of
differentiability coincides with the concept of
$P$-differentiability given by P. Pansu in \cite{pan}.

The following assertion is similar to the corresponding statement
of \cite[Proposition 2.3]{v1}.

 \begin{Proposition}\label{Prop2.1}
 Definition~{\rm{\ref{defdiff}}}  is equivalent to each
of the following assertions:

$1)$ $  d^{f(g)}\bigl(\tilde\delta^{f(g)}_{t^{-1}}
f\bigl(\delta^g_{t}(v)\bigr),L(v)\bigr)= o(1)$ as $t\to0$, where
$o(\cdot)$ is uniform in the~points $v$ of any compact part of
$\mathcal G^g\mathbb X;$

$2)$ $ d^{f(g)}(f(v),L(v))=o\bigl(d_{\mathbb X}(g,v)\bigr)$
  as  $E\cap {\mathcal G^g\mathbb X}\ni v\to g;$

$3)$ $d_{\mathbb Y}(f(v),L(v))=o\bigl(d^g(g,v)\bigr)$
  as  $E\cap {\mathcal G^g\mathbb X}\ni v\to g;$

$4)$ $ d_{\mathbb Y} (f(v),L(v))=o\bigl(d_{\mathbb X}(g,v)\bigr)$
  as  $E\cap {\mathcal G^g\mathbb X}\ni v\to g;$

$5)$  $  d_{\mathbb Y}\bigl(f\bigl(\delta^g_{t}(v)\bigr),L\bigl(
\delta^{g}_{t}v\bigr)\bigr)= o(t)$ as $t\to0$, where $o(\cdot)$ is
uniform in the points $v$ of any compact part of $\mathcal
G^g\mathbb X$.
\end{Proposition}

\begin{proof}
Consider a point $v$ of a compact part of  $\mathcal G^g\mathbb X$
and a sequence $\varepsilon_i\to0$ as $i\to0$ such that
$\delta^g_{\varepsilon_i}v\in E$ for all  $i\in \mathbb N$.
From~\eqref{estdiff} we have $
d^{f(g)}\bigl(f\bigl(\tilde\delta^g_{\varepsilon_i}v\bigr),
 L\bigl(\tilde\delta^g_{\varepsilon_i}v\bigr)\bigr)=
 o\bigl(d^g\bigl(g,\delta^g_{\varepsilon_i}v\bigr)\bigr)
= o({\varepsilon_i})$. In view of~\eqref{commut}, we infer
 $$ d^{f(g)}\bigl(\tilde\delta^{f(g)}_{\varepsilon_i}
\bigl(\tilde \delta^{f(g)}_{\varepsilon_i^{-1}}f
\bigl(\delta^g_{\varepsilon_i}v\bigr)\bigr),\tilde\delta^{f(g)}_{\varepsilon_i}L(v)
\bigr)=o(\varepsilon_i)\quad \text{uniformly in~$v$.}$$ From here,
applying the cone property of Proposition \ref{cone_prop}, we
obtain just item~1. Obviously, the arguments are reversible.
Item~1 is equivalent to item~5 since in view of \eqref{a3} we have
\begin{multline}\label{5.3}
\bigl|
  d_{\mathbb Y}\bigl(\tilde\delta^{f(g)}_{\varepsilon_i}
\bigl(\tilde \delta^{f(g)}_{\varepsilon_i^{-1}}f
\bigl(\delta^g_{\varepsilon_i}v\bigr)\bigr),\tilde\delta^{f(g)}_{\varepsilon_i}L(v)
\bigr)-
  d^{f(g)}\bigl(\tilde\delta^{f(g)}_{\varepsilon_i}
\bigl(\tilde \delta^{f(g)}_{\varepsilon_i^{-1}}f
\bigl(\delta^g_{\varepsilon_i}v\bigr)\bigr),\tilde\delta^{f(g)}_{\varepsilon_i}L(v)
\bigr)\bigr|\\
=\bigl|
 d_{\mathbb Y}\bigl(\tilde\delta^{f(g)}_{\varepsilon_i}
\bigl( \tilde\delta^{f(g)}_{\varepsilon_i^{-1}}f
\bigl(\delta^g_{\varepsilon_i}v\bigr)\bigr),\tilde\delta^{f(g)}_{\varepsilon_i}L(v)
\bigr)- o(\varepsilon_i)
  \bigr|
 =o(\varepsilon_i)\quad \text{uniformly in~$v$.}
 \end{multline}

Item 5 implies item 3 and vice versa. By comparing the metrics:
$d^{g}(g,v)=O\bigl( d_{\mathbb X}(g,v) \bigr)$ and $ d_{\mathbb
X}(g,v) =O\bigl(d^{g}(g,v)\bigr)$,
 we obtain the equivalence of the items~3 and 4. The proof of an equivalence
 of the items 4 and 2 is similar to \eqref{5.3}.
\end{proof}

Let us generalize the chain rule of paper \cite{v1}.

\begin{Theorem}\label{chainrule}
Suppose that $\mathbb X, {\mathbb Y},\mathbb Z$ are three
quasimetric spaces with strong uniform nondegenerate dilation
structures, $E$ is a set in $\mathbb X$, and $f:E\to{\mathbb Y}$
is a mapping from $E$ into $\mathbb Y$ $\delta$-differentiable at
a point $g\in E$. Suppose also that $F$ is a set in $\mathbb Y$,
$f(E)\subset Y$ and $\varphi:F\to \mathbb Z$ is a mapping from $F$
into $\mathbb Z$\,\, $\tilde\delta$-differentiable at
$p=f(g)\in{\mathbb Y}$. Then the composition $\varphi\circ
f:E\to{\mathbb Z}$ is $\delta$-differentiable at $g$ and
$$
D(\varphi\circ f)(g)=D\varphi(p)\circ Df(g).
$$
\end{Theorem}

\begin{proof} By hypothesis,
$d^{f(g)}(f(v),Df(g)(v))=o\bigl(d^g(g,v)\bigr)$ as $v\to g$ and
also $d^{\varphi(p)}(\varphi(w)$, $D\varphi(p)(w))=o\bigl(
d^p(p,w)\bigr)$
 as $w\to p$.  It follows that $f$ is continuous in $g\in E$
 and $\varphi$ is continuous in $p\in F$. We now infer
  \begin{multline*}
d^{\varphi(p)}((\varphi\circ f)(v),(D\varphi(p)\circ Df(g))(v))\\
\leq Q_{\mathbb
Z}\bigl[d^{\varphi(p)}(\varphi(f(v)),D\varphi(p)(f(v))) +
d^{\varphi(p)}(D\varphi(p)(f(v)),D\varphi(p)(Df(g)(v)))\bigr]\\
\leq o\bigl( d^p(p,f(v))\bigr)+O\bigl(
d^p\bigl(f(v),Df(g)(v)\bigr)\bigr)
\\
\leq o\bigl(d^g(g,v)\bigr)+O\bigl(o\bigl(d^g(g,v)\bigr)\bigr)=
o\bigl(d^g(g,v)\bigr)\quad\text{as $v\to g$},
\end{multline*}
since
  \begin{multline*}
d^p\bigl(p,f(v)\bigr)\leq Q_{\mathbb
Y}\left[d^p\bigl(p,Df(g)(v)\bigr)+
d^p\bigl(f(v),Df(g)(v)\bigr)\right]
\\ =O\bigl(d^g(g,v)\bigr)+o\bigl(d^g(g,v)\bigr)=O\bigl(d^g(g,v)\bigr)
\quad \text{as } v\to g.
\end{multline*}
(The estimate $d^p\bigl(p,Df(g)(v)\bigr)=O\bigl(d^g(g,v)\bigr)$ as
$v\to g$ follows from the continuity of the homomorphism $Df(g)$
and~\eqref{commut}.)
\end{proof}

\begin{Remark}\label{comp}
Note that the concept of differentiability for the quasiconformal
mappings of Carnot-Carath\'eodory manifolds was first suggested by
Margulis and Mostow in~\cite{mm} and is essentially based on
Mitchell's paper~\cite{mit}: {\it A quasiconformal mapping
$\varphi:{\mathbf M}\to{\mathbf N}$ is differentiable at a point $x_0$
in the sense of~{\rm\cite{mm}}
 if the family of mappings $\varphi_t:({\mathbf M},td_{\mathbf
M})\to ({\mathbf N},td_{\mathbf N})$ induced by the mapping
$\varphi:({\mathbf M},d_{\mathbf M})\to ({\mathbf N},d_{\mathbf N})$
converges to a horizontal homomorphism of the tangent cones at the
points $x_0\in{\mathbf M}$ and $\varphi(x_0)\in{\mathbf N}$ as
$t\to\infty$ uniformly on compact sets.} Unfortunately, this
definition is not well suitable for studying the differentials.
The problem is that the tangent cone is a class of isometric
spaces. Dealing with differentials, one would prefer to know what
happens in a fixed direction of a tangent space. In this context,
in applications of differentials it is important to know how a
concrete representative of the tangent cone is geometrically and
analytically connected with the given (quasi)metric space.
   \end{Remark}


\begin{thebibliography}{asa67}

\bibitem{as}  Agrachev A.A., Sachkov Yu.L. Control theory from the geometric
viewpoint. 2004.

\bibitem{bel}  Bellaiche A. The tangent space in sub-Riemannian geometry.
Sub-Riemannian Geometry, Progress in Mathematics, 144.
Birckh\"auser, 1996. pp. 1--78.

\bibitem{ber} Berestovskii V. N. Homogeneous manifolds with an
intrinsic metric. I. Sibirsk. Mat. Zh. \textbf{29} (6)  (1988)
17--29.

\bibitem{bram}
Bramanti~M., Brandolini~L., Pedroni~M.  Basic properties of
nonsmooth H\"{o}rmander vector fields and Poincar\'{e}s
inequality. (2009) arXiv:0809.2872.

\bibitem{lan} Bongfioli A., Lanconelli E., Uguzzoni F. Stratified
Lie groups and potential theory for their sub-laplacians.
Springer-Verlag, Berlin-Heidelberg, 2007.

\bibitem{bulf} Buliga M. Dilatation structures I. Fundamentals. J.
Gen. Lie Theory Appl. \textbf{1} (2)  (2007)   65--95.

\bibitem{bul_contrac} Buliga M. Contractible groups and linear dilatation structures.
(2007) arxiv.org: 0705.1440v3.

\bibitem{bul_sr1} Buliga M. Dilatation structures in
sub-Riemannian geometry. (2007)
 arxiv.org: 0708.4298.

\bibitem{bul_sr} Buliga M. A characterization of sub-Riemannian
spaces as length dilatation structures cunstructed via coherent
projections. (2009) arxiv.org: 0810.5040v3.

\bibitem{bul_sr_sym} Buliga M. Braided space with dilations and sub-Riemannian
symmetric spaces. (2010) arxiv.org: 1005.5031v1.

\bibitem{bur}
{Burago D.~Yu., Burago Yu.~D., Ivanov S.~V.} A Course in Metric
Geometry. Graduate Studies in Mathematics, {\bf 33}, American
Mathematical Society, Providence, RI, 2001.

\bibitem{cdpt}
Capogna L., Danielli D., Pauls S.~D. and Tyson J.~T. An
introduction to the Heisenberg group and the sub-Riemannian
isoperimetric problem. Progress in Mathematics \textbf{259}.
Birkh\"{a}user, 2007.


\bibitem{ch} Cheeger J. Differentiability of Lipschitz functions on
metric measure spaces. Geom. Funct. Anal. \textbf{9} (3) (1998)
428--517.

\bibitem{christ} Christ M. Lecture on singular operators. CBMS
Reg. Conf. er. Math., Vol. 77, Amer. Math. Soc., Providence, RI,
1990.

\bibitem{ld} Le Donne E. Geodesic manifolds with a transitive
subset of smooth bilipschitz maps. arxiv.org: 0804.0403v1

\bibitem{gol_dries} Van der Dries L., Goldbring I.
Locally compact contractive local groups.
 J. of Lie Theory. \textbf{19} (2010)  685--695.

\bibitem{fol} Folland G. B. Applications of analysis on nilpotent
groups to partial differential equations. Bull. of Amer. Math.
Soc. Vol. 83, No. 5 (1977) 912--930.

\bibitem{fs} Folland G. B., Stein E. M. Hardy spaces on homogeneous
groups. Princeton Univ. Press, 1982.

\bibitem{gleas} Gleason A. M. Groups without small subgroups. Ann.
of Math.  \textbf{56} (1952)  193--212.


\bibitem{gol} Goldbring I. Hilbert's fifth problem for local
groups. J. of Logic and Analysis. \textbf{1:5} (2009), 1--25.

\bibitem{gre} Greshnov A. V. Local approximation of equiregular
 Carnot-Carath\'eodory spaces by its tangent cones. Sib. Math. Zh.
\textbf{48} (2)  (2007)  290--312.

\bibitem{gre1}  Greshnov A. V. Applications of the group analysis
 of differential equations to some
systems of noncommuting $C^1$-smooth vector fields. Sibirsk. Mat.
Zh. \textbf{50:1} (2009), 47--62.

\bibitem{gro1} Gromov M.  Groups of polynomial growth and expanding maps.
Inst. Hautes Etudes Sci. Publ. Math. \textbf{53} (1981)  53--73.

\bibitem{gro2}  Gromov M. Metric
Structures for Riemannian and Non-Riemannian Spaces.
Birkh\"{a}user, 2001.


\bibitem{gro} Gromov M.  Carno--Carath\'eodory spaces seen from within.
 Sub-riemannian Geometry, Progress in Mathematics, \text{144}.
Birckh\"auser. (1996) 79--323.

\bibitem{hei} Heinonen J. Lectures on analysis on metric spaces.
Universitext, Springer-Verlag, New York, 2001.

\bibitem{hor} H\"{o}rmander L. Hypoelliptic second order
differential equations. Acta Math. \textbf{119} (3-4) (1967)
147--171.

\bibitem{karm} Karmanova M. A New Approach to Investigation
of Carnot-Caratheodory Geometry. Doklady Mathematics \textbf{433}
(4) (2010), to appear.

\bibitem{vk} Karmanova M., Vodopyanov S. Geometry of Carno-Carath\'eodory
spaces, differentiability, coarea and area formulas.  Analysis and
Mathematical Physics. Trends in Mathematics,  Birckh\"auser.
(2009)
 233--335.

\bibitem{macseg} Mac\`{\i}as R. A., Segovia C. Lipshitz
functions on spaces of homogeneous type. Adv. in Math. \textbf{33}
(1979) 257--270.

\bibitem{mal}  Mal'tsev A. I. On local and global topological groups.
Dokl. Akad. Nauk SSSR. \textbf{32} (9) (1941) 606--608.

\bibitem{mm} Margulis G. A., Mostov G. D. The differential of
quasi-conformal mapping of a Carnot-Caratheodory spaces. Geom.
Funct. Anal. \textbf{5} (2) (1995)  402--433.

\bibitem{mm1} Margulis G. A., Mostov G. D. Some remarks on definition of tangent
cones in a  Carnot-Caratheodory space. J.  Anal. Math. \textbf{80}
(2000)
 299--317.

\bibitem{mit} Mitchell J. On Carnot-Caratheodory metrics. J.
Differential Geometry \textbf{21} (1985)  35--45.

\bibitem{morbid}  Montanari A.,  Morbidelli D.
Balls defined by nonsmooth vector fields and the Poincare'
inequality.  Annales de l'institut Fourier. \textbf{54}(2) (2004)
431--452.

\bibitem{M} R. Montgomery. A Tour of Subriemannian Geometries, their
Geodesics and Applications. Providence, AMS. 2002.

\bibitem{mz} Montgomery D., Zippin L. Topological transformation
groups. Interscience, New York. 1955.

\bibitem{mr} M\"{u}ller-R\"{o}mer P. Kontrahierbare Erweiterungen
kontrahierbaren Gruppen. J. Reine Angew. Math. \textbf{283/284}
(1976)
 238--264.

\bibitem{nsw} Nagel A., Stein E.M., Wainger S. Balls and metrics
defined by vector fields I: Basic properties. Acta Math.
\textbf{155} (1985) 103--147.

\bibitem{olv} Olver P. Non-associative local Lie groups. Journal
of Lie theory \textbf{6} (1996)  23--51.

\bibitem{PS} Paluszy$\acute{\text{n}}$ski M., Stempak K. On quasi-metric and metric
spaces // AMS Proccedings. \textbf{137} (12) (2002)  4307--4312.

\bibitem{pan} Pansu P. Metriques de Carnot-Carath\'eodory et
quasiisometries des espaces symetriques de rang un. Ann. of Math.
 \textbf{119} (1989) 1--60.

\bibitem{pet} Petersen V. P. Gromov--Hausdorff convergence in
metric space. Differential geometry: Riemannian geometry (Proc.
Sympos. Pure Math., \textbf{54} Pt.3). Providence, RI: Amer. Math.
Soc. (1993) 489--504.

\bibitem{pon}   Pontryagin L. S. Continuous Groups.
Moscow, ''Nauka``. 1984.


\bibitem{rs} Rotshild L.P., Stein E.M. Hypoelliptic differential operators
 and nilpotent groups. Acta Math. \textbf{137} (1976)  247--320.

\bibitem{dan} Selivanova S. V. Tangent cone to a regular
quasimetric Carnot--Carath\'eodory space.  Doklady Mathematics
\textbf{79} (2009)  265--269.

\bibitem{smj} Selivanova S. V. Tangent cone to a quasimetric space with dilations.
// Sib. Mat. J.  \textbf{51} (2) (2010) 388-403.

\bibitem{sieb} Siebert E. Contractive automorphisms on locally
compact groups. Mat. Z. \textbf{191} (1986)  73--90.

\bibitem{Stein} Stein E.~M., Harmonic analysis:
 real-variables methods, orthogonality, and oscillatory integrals.
  Princeton, NJ, Princeton University Press. 1993.

 \bibitem{v}
Vodopyanov~S. K. Differentiability of mappings in the geometry  of
Carnot manifolds. Sib. Math. Zh. \textbf{48} (2)  (2007),
251--271.

 \bibitem{v1}
Vodopyanov~S. K. Geometry of Carnot--Carath\'eodory spaces and
differentiability of mappings. Contemporary Mathematics
\textbf{424} (2007), 247--302.


\bibitem{vk1}
Vodopyanov~S. K.,  Karmanova M. B. Local Geometry of Carnot
Manifolds Under Minimal Smoothness. Doklady Mathematics
\textbf{75} (2) (2007), 240--246.

\bibitem{vk2}
Vodopyanov~S. K.,  Karmanova M. B.
 Sub-Riemannian geometry
under minimal smoothness of vector fields. Doklady Mathematics
\textbf{78} (2)  (2008), 583--588.

\bibitem{vk3}
Vodopyanov~S. K.,  Karmanova M. B. A Coarea Formula for Smooth
Contact Mappings of Carnot Manifolds. Doklady Mathematics,
\textbf{76} (4)  (2007),  908--912.


\bibitem{vk4}
Vodopyanov~S. K.,  Karmanova M. B. An Area Formula for Contact
$C^1$-Mappings of Carnot Manifolds. Doklady Mathematics,
\textbf{78} (2)  (2008)  655--659.


\bibitem{vk5}
Vodopyanov~S. K.,  Karmanova M. B. An Area Formula for Contact
$C^1$-Mappings of Carnot Manifolds. Complex Variables and Elliptic
Equations. \textbf{55}(1) (2010) 317--329.

\bibitem{vs}
Vodopyanov~S. K.,  Selivanova~S. V. Algebraic properties of the
tangent cone to a quasimetric space with dilations // Doklady
Mathematics \textbf{80} (2)  (2009)  734--738.

\bibitem{hilb}  Yandell B. H.
The honor class: Hilbert's problems and their solvers. AK Peters,
Natick, Massachusetts. 2002.
\end{thebibliography}
\end{document}